\documentclass[11pt,english]{article}

\usepackage[margin = 2.35cm]{geometry}

\usepackage{amsthm}
\usepackage{amsmath}
\usepackage{amssymb}
\usepackage{setspace}
\usepackage{mathtools}
\usepackage{graphicx}
\usepackage[hidelinks]{hyperref}
\usepackage{thm-restate}
\usepackage{cleveref}
\usepackage{hyperref}
\usepackage{enumitem}
\usepackage{framed}
\usepackage{subcaption}

\usepackage[square,sort,comma,numbers]{natbib}
\usepackage{floatrow}

\theoremstyle{plain}

\newtheorem*{thm*}{Theorem}
\newtheorem{thm}{Theorem}
\Crefname{thm}{Theorem}{Theorems}

\newtheorem*{lem*}{Lemma}
\newtheorem{lem}[thm]{Lemma}
\Crefname{lem}{Lemma}{Lemmas}

\newtheorem*{claim*}{Claim}
\newtheorem{claim}[thm]{Claim}
\crefname{claim}{Claim}{Claims}
\Crefname{claim}{Claim}{Claims}

\newtheorem{prop}[thm]{Proposition}
\Crefname{prop}{Proposition}{Propositions}

\newtheorem{cor}[thm]{Corollary}
\crefname{cor}{Corollary}{Corollaries}

\crefname{conj}{Conjecture}{Conjectures}

\Crefname{qn}{Question}{Questions}

\newtheorem{obs}[thm]{Observation}
\Crefname{obs}{Observation}{Observations}

\Crefname{ex}{Example}{Examples}

\theoremstyle{definition}

\Crefname{prob}{Problem}{Problems}

\newtheorem{defn}[thm]{Definition}
\Crefname{defn}{Definition}{Definitions}

\theoremstyle{remark}

\renewenvironment{proof}[1][]{\begin{trivlist}
\item[\hspace{\labelsep}{\bf\noindent Proof#1.\/}] }{\qed\end{trivlist}}

\newcommand{\remove}[1]{}

\newcommand{\ceil}[1]{
    \lceil #1 \rceil
}
\newcommand{\floor}[1]{
    \lfloor #1 \rfloor
}

\newcommand{\eps}{\varepsilon}
\renewcommand{\P}{\mathcal{P}}

\newcommand{\cm}{connected matching}
\newcommand{\cms}{connected matchings}

\providecommand{\keywords}[1]
{
  \small	
  \textbf{\textit{Keywords:}} #1
}

\begin{document}


\title{\vspace{-0.9cm} Three colour bipartite Ramsey number of cycles and paths}

\author{
	Matija Buci\'c\thanks{
	    Department of Mathematics, 
	    ETH, 
	    8092 Zurich;
	    e-mail: \texttt{matija.bucic}@\texttt{math.ethz.ch}.
	}
	\and
    Shoham Letzter\thanks{
        ETH Institute for Theoretical Studies,
        ETH,
        8092 Zurich;
        e-mail: \texttt{shoham.letzter}@\texttt{eth-its.ethz.ch}.
        Research supported by Dr.~Max
        R\"ossler, the Walter Haefner Foundation and by the ETH Zurich Foundation.
    }
    \and
	Benny Sudakov\thanks{
	    Department of Mathematics, 
	    ETH, 
	    8092 Zurich;
	    e-mail: \texttt{benjamin.sudakov}@\texttt{math.ethz.ch}.     
		Research supported in part by SNSF grant 200021-175573.
	}
}
\date{}

\maketitle

\begin{abstract}
    \setlength{\parskip}{\smallskipamount}
    \setlength{\parindent}{0pt}
    \noindent
    
    \vspace{-0.9cm} 
    The \textit{$k$-colour bipartite Ramsey number} of a bipartite graph $H$ is the least integer $n$ for which every $k$-edge-coloured complete bipartite graph $K_{n,n}$ contains a monochromatic copy of $H$. The study of bipartite Ramsey numbers was initiated, over 40 years ago, by Faudree and Schelp and, independently, by Gy\'arf\'as and Lehel, who determined the $2$-colour Ramsey number of paths. In this paper we determine asymptotically the $3$-colour bipartite Ramsey number of paths and (even) cycles.
    
\end{abstract}

\section{Introduction}
	Ramsey theory refers to a large body of mathematical results, which roughly say that any sufficiently large structure is guaranteed to have a large well-organised substructure. For example, the celebrated theorem of Ramsey \cite{ramsey1929problem} says that for any fixed graph $H$, every $k$-edge-colouring of a sufficiently large complete graph contains a monochromatic copy of $H$. The \emph{$k$-colour Ramsey number of $H$}, denoted $r_k(H)$, is defined to be the smallest order of a complete graph satisfying this property.
	
	Despite significant attention paid to Ramsey problems, there are very few examples of families of graphs whose Ramsey numbers are known exactly, or even just asymptotically. An early example of an exact Ramsey result was obtained in 1967 by Gerencs\'er and Gy\'arf\'as \cite{ramsey2path}, who determined the $2$-colour Ramsey number of paths, showing that $r_2(P_n) = \floor{3n/2 - 1}$, where $P_n$ is the path on $n$ vertices. Faudree and Schelp \cite{faudree-schelp-cycles} and, independently, Rosta \cite{rosta-cycles-73} later determined the $2$-colour Ramsey number of cycles. The $3$-colour case was much more difficult and it took 25 more years until \L{}uczak \cite{luczak99-con-match} determined it asymptotically for odd cycles, showing that $r_3(C_n) = 4(1 +o(1))n$. In his paper \L{}uczak introduced a technique that uses the Szemer\'edi's regularity lemma to reduce problems about paths and cycles to problems about \emph{connected matchings}, which are matchings that are contained in a connected component. This technique has become fairly standard in the area, and, indeed, many of the results that we describe here as well as our own results make use of it. The $3$-colour Ramsey numbers for paths and even cycles were determined, asymptotically, by \L{}uczak and Figaj \cite{figaj-luczak}. These results were strengthened to exact result for long odd cycles by Kohayakawa, Simonovits and Skokan \cite{kohayakawa-simonovits-skokan}, for long paths by Gy\'arf\'as, Ruszink\'o, S\'ark\"{o}zy and Szemer\'edi \cite{ramsey3path} and for long even cycles by Benevides and Skokan \cite{benevides-skokan}. The odd cycles result was recently generalised to $k$ colours by Jenssen and Skokan \cite{jenssen-skokan} who proved that for every $k$ and sufficiently large \emph{odd} $n$, $r_k(C_n) = 2^{k-1}(n-1)+1$. Interestingly, this  does not hold for all $k$ and $n$, as was shown by Day and Johnson \cite{day-johnson}. Ramsey numbers of paths and even cycles are not understood as well for $k \ge 4$. The best known bounds (for paths as well as even cycles) are $(k-1-o(1))n\le r_k(P_n) \le (k-1/2 + o(1))n$. The lower bound is due to Yongqi, Yuansheng, Feng and Bingxi \cite{Yongqi}, and the upper bound is due to Knierim and Su \cite{knierim}.

	Over the years, many generalisations of Ramsey numbers have been considered (the survey \cite{conlon2015recent} contains many examples); one natural example is obtained by replacing the underlying complete graph by a complete bipartite graph. In particular, the \textit{$k$-colour bipartite Ramsey number} of a bipartite graph $H$, denoted $r^{\text{bip}}_k(H)$, is the least integer $N$ such that in any $k$-colouring of the complete bipartite graph $K_{N,N}$ there is a monochromatic copy of $H$.

	The study of bipartite Ramsey numbers was initiated in the early 70s by Faudree and Schelp \cite{faudree-schelp} and independently Gy\'arf\'as and Lehel \cite{gyarfas2path} who both considered the $2$-colours case for paths. They showed that
    $$
        r^{\text{bip}}_2(P_n)=
        \begin{cases}
            n-1 & \text {if } n \text{ is even},\\
            n & \text {if } n \text{ is odd.}
        \end{cases}
    $$
    The natural extension to cycles has been considered recently. Zhang and Sun \cite{zhang1} and Zhang, Sun and Wu \cite{zhang2} determine exact asymmetric $2$-colour Ramsey numbers for (even) cycles $r^{\text{bip}}_2(C_{2n},C_{2m})$ for $m \le 3$. As in the case of ordinary Ramsey numbers, few exact or asymptotic results for more colours are known, even for three colours. Joubert \cite{joubert} considers $k$-colour bipartite Ramsey number of even cycles, and obtains some bounds and exact results when all the cycles have length at most $8$. 
    
    Bipartite Ramsey numbers were also studied for the complete bipartite graphs, and the first to consider this were Beineke and Schwenk \cite{beineke} in $1976$. Similarly to the case of ordinary Ramsey numbers, the best known lower bound on $r^{\text{bip}}_k(K_{n,n})$, due to Hattingh and Henning \cite{hattingh} and the best known upper bound, due to Conlon \cite{conlon}, are still exponentially apart.
    
    In this paper we determine, asymptotically, the $3$-colour bipartite Ramsey number of paths. In fact, we determine the $3$-colour Ramsey number for the case of even cycles; the result for paths follows as a corollary.  

    \begin{thm}\label{thm:main}
        $r^{\text{bip}}_3(C_{2n}) = (3+o(1))n.$
    \end{thm}

    \begin{cor}\label{cor:main}
        $r^{\text{bip}}_3(P_{n}) = (3/2+o(1))n.$
    \end{cor}
    
    The following example provides the lower bounds for the above results.
	
	\begin{restatable}{ex}{mainExample} \label{ex:main-example}
	    Given $N=a_1+a_2+a_3$, split the vertices of the left part of $K_{N,N}$ into sets $A_1,A_2,A_3$ of sizes $|A_i|=a_i.$ Colour any edge touching a vertex of $A_i$ in colour $i.$ 
    \end{restatable}

    \begin{figure}[ht]
        \caption{Example \ref{ex:main-example}.}
        \includegraphics[scale = .8]{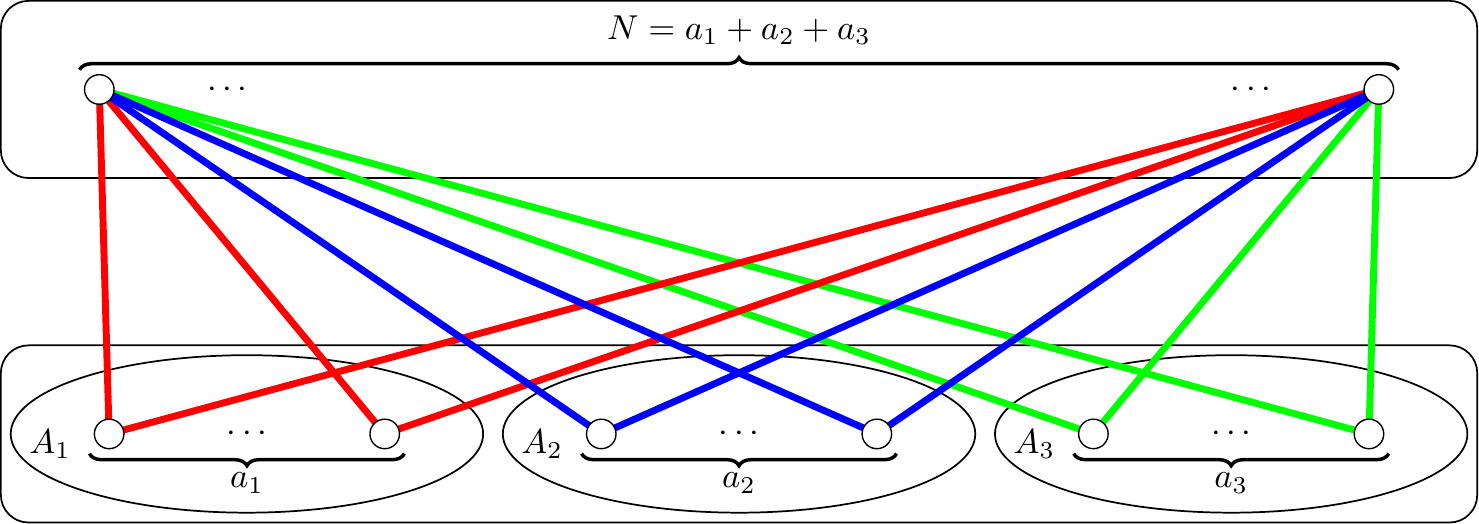}
        \label{fig:eg1}
	\end{figure}
	
    Note that in this example the longest path in colour $i$ has order $2a_i+1$, so there is no $C_{2(a_i+1)}$ or $P_{2(a_i+1)}$ in colour $i$, implying that $r^{\text{bip}}_3(C_{2n}),r^{\text{bip}}_3(P_{2n})\ge 3n-2.$
	
    \subsection{Organisation of the paper}	

        In our proof of \Cref{thm:main} we use \L{}uczak's method of converting problems about cycles and paths to problems about connected matchings. 
        The method requires us to find a monochromatic connected matching (i.e.\ a matching that is contained in a monochromatic component) in the so-called reduced graph (obtained by applying Szemer\'edi's regularity lemma), which is almost complete bipartite.
        To that end, we first obtain an exact bipartite Ramsey result for connected matchings (see \Cref{thm:three-col-symm-cm}); we do this in \Cref{sec:cm}. Many proofs that employ \L{}uczak's method work directly with an almost complete graph (or almost complete bipartite, as in our setting). However, as our proof of \Cref{thm:three-col-symm-cm} uses induction, we were unable to directly prove a variant of this theorem for almost complete bipartite graphs. Instead, we deduce the version for almost complete bipartite graphs from the theorem for complete bipartite graphs; see \Cref{sec:almost-complete}. This approach is to the best of our knowledge new and we believe it may be very useful, or at least simplify, other arguments which use the connected matchings method described above. In \Cref{sec:regularity} we introduce the preliminaries required for the application of \L{}uczak's method, including a multicolour version of the regularity lemma; we then use this method to complete the proof of \Cref{thm:main}. We conclude the paper in \Cref{sec:conc-remarks} with some remarks and open problems.

	    Throughout the paper we think of the two parts in the bipartition of a bipartite graph as the left and right hand sides. Given a bipartite graph $G$ we denote its left hand side by $L(G)$ and its right hand side by $R(G)$; when this is not likely to cause confusion, we omit $G$ from this notation. We call the three colours in a $3$-coloured graph red, green and blue. We denote the degree of a vertex $v$ in colour $C$ by $d_C(v)$ and the set of vertices joined by an edge of colour $C$ to $v$ as $N_C(v)$.
	
\section{Monochromatic \cms{} in \texorpdfstring{$K_{n,n}$}{Kn,n}}\label{sec:cm}
        
	Given an edge-coloured graph $H$, a \textit{$C$-coloured connected matching} is a matching that is contained in a connected component in the graph spanned by the edges of colour $C$; a \textit{$k$-\cm{}} is a \cm{} of size $k$. Let $r(k,l,m)$ denote the smallest integer $n$ such that for any $3$-colouring of $K_{n,n}$ there is a red $k$-\cm{}, a green $l$-\cm{} or a blue $m$-\cm{}. In the following theorem we determine $r(k,k,k)$.
	
	\begin{restatable}{thm}{thmThreeColSymmCM} \label{thm:three-col-symm-cm}
		$r(k,k,k) = 3k-2.$
	\end{restatable}
    \begin{proof}
        Note that \Cref{ex:main-example} implies that $r(k,k,k) \ge 3k-2$. We prove the upper bound by induction on $k$; in other words, we prove by induction that every $3$-colouring of $K_{n,n}$, where $n=3k-2$, contains a monochromatic $k$-\cm.
        
        For the basis we note that if $k=1$ this statement is equivalent to the existence of an edge, which is true as $n = 3k-2 = 1$. We also show that when $k = 2$ and $n = 3k-2 = 4$ there is a monochromatic $2$-\cm{} (note that having a monochromatic $2$-\cm{} is equivalent to having a monochromatic $P_4$ so this will imply that $r^{\text{bip}}_3(P_{4}) = 4.$) Towards a contradiction, suppose that there is a $3$-colouring of $K_{4,4}$ without monochromatic $2$-\cms. Since every vertex has degree $4$, it has degree at least $2$ in some colour, and since there are four vertices on the left, there are two vertices, $u_1$ and $u_2$, with degree at least $2$ in some colour, say red. If the red neighbourhood of $u_1$ and $u_2$ intersect, then there is a red $2$-\cm. Otherwise, their red neighbourhoods cover the right hand side, so the existence of any other red edge not incident with $u_1$ or $u_2$ would imply the existence of a red $2$-\cm{}. Hence, we may assume that the two remaining vertices on the left, $u_3$ and $u_4$, have red degree $0$, and all vertices on the right have red degree $1$. An analogous argument implies the existence of two vertices on the right, $v_1$ and $v_2$, whose degree in a colour other than red, say blue, is $0$. It follows that all edges between $\{u_3, u_4\}$ and $\{v_1, v_2\}$ are green; in particular, there is a green $2$-\cm.
        
        We now assume $r(k,k,k) \le 3k-2$, where $k \ge 2$, and our aim is to prove that $r(k+1, k+1, k+1) \le 3k+1$. In order to prove the latter statement, we first prove that $r(k,k+1,k+1) \le 3k$, and then use this inequality to prove the desired statement. 

        The following observation follows easily from K\"onig's theorem \cite{konig}, which states that in bipartite graphs the size of a maximum matching is the same as the size of a minimum cover; we shall use this observation throughout the proof of \Cref{thm:three-col-symm-cm}.
        
        \begin{obs} \label{obs:konig}
            Let $H$ be a connected bipartite graph whose maximum matchings have size $l$. Let $v$ be a vertex that is contained in some minimum cover of $H$. Then a maximum matching in $H \setminus \{v\}$ has size $l-1$.
        \end{obs}
        
        We now introduce some notation that we will need for the proofs of the two statements: $r(k,k+1,k+1) \le 3k$ and $r(k+1,k+1,k+1) \le 3k+1$. In each of these proofs we assume, for the sake of contradiction, that there is a colouring of $K_{n,n}$ (where $n=3k$ and $n=3k+1$ respectively) without red \cms{} of size $k$ and $k+1$, respectively, and without blue or green \cms{} of size $k+1$.
            
        We call a monochromatic component \emph{big} if it contains a matching of size $k$. As the number of vertices in each side is at most $3k+1$ and $k \ge 2$, there are at most three big components in each colour. 
        We will use the following properties of big components.
        \begin{enumerate}[label=\textbf{B\arabic{*}}., ref=(\textbf{B\arabic{*}})]
            \item \label{itm:b1} Every big component has a minimum cover of size $k$,
            \item \label{itm:b2} every big component has at least $k$ vertices in each side,
            \item \label{itm:b3} if there is a cover of a big component $C$ which is contained in one of the sides, then there are no other vertices of $C$ in that side. 
        \end{enumerate}
        Property \ref{itm:b1} follows by K\"onig's theorem and the assumption that there are no monochromatic \cms{} of size $k+1.$ Property \ref{itm:b2} follows as big component contains a matching of size $k.$ Property \ref{itm:b3} follows as if there were vertices other than the ones in the cover, on the same side, their edges (which exist, as the component is connected) would not be covered by this cover.
        
        Given a monochromatic component $C$, we call a vertex $v$ a \emph{cover vertex} of $C$ if there is a minimum cover of $C$ that contains $v$. We say that a big component is of type $L$ if it contains a cover vertex in the left side, and we say that it is of type $R$ if it contains a cover vertex in the right. We note that a component can be of both types, in which case we say that it is \emph{unspecified}, or it can be of only one type, in which case we call it \emph{specified}. 
            
        We will use the following simple properties of specified and unspecified components.
        \begin{enumerate}[label=\textbf{S\arabic{*}}., ref=(\textbf{S\arabic{*}})]
            \item \label{itm:s1} Every specified component has exactly $k$ vertices in one side, and these constitute a minimum cover,
            \item \label{itm:s2} given a big component with exactly $k$ vertices in each side, each of its vertices is a cover vertex, hence the component is unspecified,
            \item \label{itm:s3} every specified component has at least $k+1$ vertices in one of the sides.
        \end{enumerate}
        To see why property \ref{itm:s1} holds, consider a specified component $C$, say of type $L$, so all cover vertices are in $L$. Then by \ref{itm:b1} there is a cover of size $k$ in $L$, hence property \ref{itm:s1} follows from \ref{itm:b3}. For property \ref{itm:s2}, note that each side of the component is a minimum cover, by \ref{itm:b1}, which implies that every vertex in the component is a cover vertex. Property \ref{itm:s3} follows directly from \ref{itm:s1} and \ref{itm:s2}.
        
        \begin{prop}\label{prop:induction-asym-three-paths}
            $r(k,k+1,k+1) \le 3k.$ 
        \end{prop}
        
        \begin{proof}
            We assume the opposite, i.e.\ that there is a $3$-colouring of $K_{n,n}$ without a red $k$-\cm{} or a blue or green $(k+1)$-\cm{}, where $n = 3k$. In particular, there are no big red components. We will need the following claim.
                        
            \begin{claim}\label{claim:three-covers} 
                If there are minimum covers of three distinct big components, all contained in the same side, then the three corresponding components have the same colour.
            \end{claim}
            
            \begin{proof}
                Assume the opposite; then, without loss of generality, there are two big blue components $B_1, B_2$ and a big green component $G_1$, all of which have a cover of size $k$ in $R$ (recall that $\{L, R\}$ is the bipartition of our complete bipartite graph). In particular, $|B_1 \cap R|=|B_2 \cap R|=|G_1 \cap R|=k$, by \ref{itm:b3}. 
                
                We note that every vertex $v \in L$ has blue degree at most $k$. Indeed, $v$ sends blue edges to at most one of the sets $B_1 \cap R$, $B_2 \cap R$ and $R \setminus (B_1 \cup B_2)$, all of which have size $k$.
                
                Let $U$ be the set of vertices $v \in L$ that have green degree at least $k+1$. Note that by \ref{itm:b3} the vertices in $U$ send green edges only to $R \setminus G_1$. As $|R \setminus G_1|=2k$, every two vertices in $U$ have a common green neighbour, so they all belong to the same green component $G_2$. Hence, if $|U| \ge k+1$, then by a greedy argument $G_2$ contains a green matching of size $k+1$, so there is a green $(k+1)$-\cm{}, a contradiction. So, $|U| \le k$.
                Let $W$ be the set of vertices in $L$ with red degree at least $k$. 
                By the above arguments, the at least $2k$ in $L \setminus U$ have both blue and green degree at most $k,$ so they have red degree at least $k$, meaning $|W| \ge 2k$.
                Then, since no red component contains a matching of size $k$ (by our initial assumption), a greedy argument shows that every red component contains at most $k-1$ vertices from $W$. Hence, $W$ consists of vertices from at least three different red components, each of which has at least $k$ vertices in $R$, because each vertex in $W$ has red degree at least $k$. Since $|R| = 3k$, this implies that the vertices of $R$ belong to exactly three red components, each of which has exactly $k$ vertices in $R$. It follows that all vertices in $L$ have red degree at most $k.$ Furthermore, since each of the three red components has at most $k-1$ vertices in $W$, we find that $|W| \le 3(k-1),$ so $L \setminus W$ is non-empty. Let $u \in L \setminus W$; so $u$ has red degree strictly smaller than $k$.
                
                Since both red and blue degrees of every vertex in $L$ are at most $k$, we conclude that the green degree of every vertex in $L$ is at least $k$. Furthermore, $u$ has green degree at least $k+1,$ so as $|G_1 \cap R| = k$ we have $u \notin G_1$. Since all vertices of $L$ have green degree at least $k,$ any $w \in L \setminus G_1$ has a common green neighbour with $u$, so in particular belongs to the same green component as $u,$ denoted by $G_2$. Note that this implies that $G_1$ and $G_2$ cover $L.$ Moreover, if $|G_2 \cap L| \ge k+1$ a greedy argument shows that $G_2$ contains a matching of size $k+1$, by picking, one by one, an unused green neighbour of a vertex in $G_2 \cap L$, letting $u$ be considered last. This is a contradiction, thus $|G_2 \cap L| \le k$ so $|G_1 \cap L| \ge 2k$.
                
                Denote $S_L = G_1 \cap L$ and $S_R = R \setminus G_1$. Then $|S_R|, |S_L| \ge 2k$, and there are no green edges between the two sets. Since every vertex in $S_L$ has blue degree at most $k$, it follows that every vertex in $S_L$ sends at least $k$ red edges into $S_R$. This implies vertices of $S_L$ belong to at most two different red components. Therefore, one of them contains at least $|S_L|/2 \ge k$ vertices from $S_L$, each of which has red degree at least $k$, which implies the existence of a red $k$-\cm{}, a contradiction.
            \end{proof}
            
            By removing any four vertices, two from each side, we obtain a $3$-coloured $K_{3k-2,3k-2}$, so by the inductive assumption we can find a monochromatic $k$-\cm{} and in particular a big component. 
            Let $C$ be a big monochromatic component and let $v$ be a cover vertex of $C$. Note that by \Cref{obs:konig}, the size of a maximum matching in $C \setminus \{v\}$ is smaller than the size of a maximum matching in $C$, hence $C \setminus \{v\}$ is not big. If we remove a vertex from the graph which is a cover vertex of a big monochromatic component $C$, we say that we removed $C$. Our goal is to show that we can remove at most two vertices from each side, in such a way that we remove all big components, thus reaching a contradiction and proving \Cref{prop:induction-asym-three-paths}.
            
            If there is a colour, say blue, with three big components, as there are $3k$ vertices on each side and each big component has at least $k$ vertices on both sides \ref{itm:b2} then each blue component has exactly $k$ vertices on each side. This implies, by \ref{itm:s2}, that every vertex is a cover vertex of its blue component and they are all unspecified. By \Cref{claim:three-covers} no minimum cover of  a big green component is contained in one side. This implies, by \ref{itm:s1}, that all green components are also unspecified. Furthermore, there are at most two big green components, since if there were three, by the same argument as for blue above, every component would have exactly $k$ vertices in each side, which implies that there is a minimum cover of a big green component which is contained in one side, a contradiction. We first remove an arbitrary cover vertex of some green big component; this also removes one of the blue big components, leaving us with at most three big components. The remaining three components are unspecified, so we simply remove a cover vertex of each one, in such a way that exactly two vertices are removed from each side. We thus obtain a $3$-colouring of $K_{3k-2,3k-2}$ without monochromatic \cms{} of size $k$, a contradiction to the induction hypothesis.

            By symmetry, we may assume that there are at most two big blue components and at most two big green components. By \Cref{claim:three-covers}, there are no three specified big components of the same type. So, if we first remove a cover vertex of each specified component, then we can remove a cover vertex from each of the remaining, unspecified, components, in such a way that we remove at most two vertices from each side. This is, again, a contradiction, and we have thus completed the proof of \Cref{prop:induction-asym-three-paths}.
        \end{proof}
            
        We now proceed with the second part of the argument, where our goal is to show that $r(k+1,k+1,k+1) \le 3k+1.$ We assume, for the sake of contradiction, that there is a colouring of $K_{n,n},$ with $n=3k+1,$ without a monochromatic \cm{} of size $k+1.$ We claim that one of the following two cases holds for each colour $C$.
        \begin{enumerate}[label=(\alph*), ref=\alph*]
            \item \label{itm:three-big}
                there are three big components in colour $C$,
            \item \label{itm:two-big}
                there are exactly two big components in colour $C$, both of which are specified and of the same type.
        \end{enumerate}
        Towards contradiction, assume that neither (\ref{itm:three-big}) nor (\ref{itm:two-big}) holds for, say, $C$ being red. As there are at most three big red components and (\ref{itm:three-big}) does not hold, there are at most two big red components. If there is only one such component, we can remove a cover vertex from it and remove another vertex from the other side, thus obtaining a $3$-colouring of $K_{3k,3k}$ without red \cms{} of size $k$ and without blue or green \cms{} of size $k+1$, a contradiction to \Cref{prop:induction-asym-three-paths}. Similarly, if there are exactly two red components $R_1$ and $R_2$, as (\ref{itm:two-big}) does not hold, we have, say, that $R_1$ is of type $R$ and $R_2$ is of type $L$ (recall that unspecified components are of both types), so we can remove one cover vertex of $R_1$ from $R$ and one cover vertex of $R_2$ from $L$, thus reaching a contradiction as before.
        
        If (\ref{itm:three-big}) holds for red, then there are at most $k$ vertices on each side whose red degree is at least $k+1$. To see this notice that if there were $k+1$ such vertices in $L$, in case they were all in the same red component, by the greedy argument, there would be a $k+1$ red \cm{}, a contradiction. Otherwise, there would be two red components with at least $k+1$ vertices in $R$, which implies that any remaining big red components has at most $3k+1-2(k+1)=k-1$ vertices in $R,$ so by \ref{itm:b2} it is not big, again a contradiction.
        
        Similarly, if (\ref{itm:two-big}) holds for red, and the two big red components are of type $R$, then at most $k-1$ vertices in $L$ have red degree at least $k+1$. Indeed, each big red component contains at least $k+1$ vertices from $L$, by \ref{itm:s3}, so there are at most $3k+1-2(k+1)=k-1$ vertices in $L$ which are not in any red big component. Furthermore, vertices in $L$ that are contained in big red components have red degree at most $k$ by \ref{itm:s1}.
        
        As every vertex has degree at least $k+1$ in some colour and there are $3k+1$ vertices in $R$ there is a colour, say red, such that $k+1$ vertices in $R$ have red degree at least $k+1$. So by the above (\ref{itm:two-big}) holds for red and the two big red components are of type $R$. By repeating this for the other side we find that there is another colour, say green, so that (\ref{itm:two-big}) holds for green and the two green components have type $L$. So, from now on we assume that (\ref{itm:two-big}) holds for both red and green, with both big red components being of type $R$ and both big green components of type $L$. We will distinguish between two cases, depending on whether blue satisfies (\ref{itm:three-big}) or (\ref{itm:two-big}).  

        Our goal now is to remove up to three vertices from each side in such a way that all big components are removed. This would imply the existence of a $3$-colouring of $K_{3k-2,3k-2}$ without monochromatic $k$-\cms, a contradiction to the inductive assumption.
        
        Suppose first that (\ref{itm:three-big}) holds for blue. Let the three big blue components be $B_1,B_2,B_3.$ As there are $6k+2$ vertices in total and, and by \ref{itm:b2} each big blue component has size at least $2k,$ we distinguish two cases, either $|B_1|=|B_2|=2k$ or $|B_1|=2k, |B_2|=|B_3|=2k+1.$ Note that $|B_i|=2k$ implies, by \ref{itm:b2}, that $B_i$ has exactly $k$ vertices on each side which are all cover vertices, so by \ref{itm:s2} $B_i$ is unspecified. In the second case, as $|R|=|L|=3k+1$ both $B_2$ and $B_3$ have exactly $k$ vertices on different sides, say $|B_2 \cap R|=k$ and $|B_3 \cap L|=k$. This means that in both cases each of $B_1$ and $B_2$ has $k$ cover vertices in $R$. Recall that there are two red specified big components of type $R$, hence together they have $2k$ cover vertices in $R$, by \ref{itm:s1}. It follows that there is a vertex in $R$ which is a cover vertex of both a big blue and a big red component. We remove this vertex from the graph, together with any cover vertex from each of the green components and any cover vertex of the remaining red component, thus removing two vertices from each of $L$ and $R.$ For the two remaining big blue components, we can remove them by removing one vertex from each side, which is possible as there is at most one blue specified component of each type (since $|B_i|=2k$ means $B_i$ is unspecified by \ref{itm:b2} and \ref{itm:s2}, so in the first case there is at most one specified blue component and in the second we have seen $B_2$ and $B_3$ need to be of distinct types). Thus, we obtain a $3$-colouring of $K_{3k-2,3k-2}$ with no monochromatic $k$-\cms{}, a contradiction.
        
        The remaining case is that (\ref{itm:two-big}) holds for blue and, without loss of generality, both big blue components are of type $R$. Now all colours satisfy (\ref{itm:two-big}), with red and blue big components being of type $R$ and green of type $L$. Then as $|R| = 3k+1$, there is a vertex that is a cover vertex of both a blue and a red big component. We remove such a vertex, as well as one cover vertex from each of the remaining big components. We reach a contradiction, as before. This completes the proof of \Cref{thm:three-col-symm-cm}.
    \end{proof}

	Let us show a simple asymmetric generalisation

    \begin{cor} \label{cor:asym-k-k->-l-cm}
        Let $k \le l$, then $r(k,l,l)=k+2l-2$.
    \end{cor}
    
    \begin{proof}
        We have $r(k,l,l) \ge k+2l-2$ by \Cref{ex:main-example}.
        
        For the upper bound, we note that the case $k=l$ follows from \Cref{thm:three-col-symm-cm} and $l=k+1$ follows from \Cref{prop:induction-asym-three-paths}. We proceed with the induction on $l$, assuming $l \ge k+2$ and using the aforementioned cases as the basis.
        
        Consider a $3$-colouring of $K_{n,n}$, where $n = k+2l$. We assume that there are no red or blue \cms{} of size $l+1$, and no green \cms{} of size $k$. We note that there are at most two big red and at most two big blue components as $3l \ge 2l+k+1$ (here we call a red or blue component \emph{big} if it has a matching of size $l$). For the same reason there are no three disjoint minimum covers in these colours that are contained on one side.
        
        We claim that it is possible to remove two vertices from each side, such that at least one cover vertex is removed from each red or blue big component. Indeed, if there are at most two specified components of each type, this holds as the total number of big components is at most four. Hence, we may assume that at least three of the big components are specified and of type $R$. If there are exactly three such components, pick one vertex that is a cover vertex of two big specified type $R$ components (this is possible by the above remark about disjoint minimum covers), then pick another cover vertex from the remaining specified type $R$ component and pick any cover vertex in $L$ of the type $L$ component (if such a component exists). Finally, suppose that there are four specified type $R$ components. Let $H$ be the auxiliary graph whose vertices correspond to the big components, and put an edge between two components if their minimum covers intersect. Note that this is a bipartite graph, each of its sides has size $2$, and every three vertices span an edge (again by the above remark). It can be easily verified that $H$ has a perfect matching, so we can pick two vertices, such that for each big component, one of these vertices is a cover vertex.
        
        We removed two vertices from each side, such that for each big component one of its cover vertices is removed. We thus obtain a $3$-colouring of $K_{k+2l-2,k+2l-2}$ without red or blue $(l-1)$-\cms{} and without green $k$-\cms{}, a contradiction to the induction hypothesis.
    \end{proof}
    
    
    We now show that there are ranges of $k,l,m$ for which the \Cref{ex:main-example} is not tight. 
    \begin{lem}   \label{lem:lower-bound} 
        Let us assume $k\ge l\ge \frac{k}{2}$ and let $t=\min\{k-l, 2l-k\}$. Then 
        $$r(k+1,l+1,l+1) > k+2l+t.$$
    \end{lem}

    \begin{proof}
        We exhibit a $3$-edge colouring of $K_{n,n}$, where $n = k+2l+t$, with no red $(k+1)$-\cm, or a blue or green $(l+1)$-\cm{}.        
        
        Let $A_1,...,A_5$ be disjoint sets of sizes: $l$, $k-l+t$, $l$, $k-l+t$, $2l-k-t$ respectively; denote $A = A_1 \cup \ldots \cup A_5$. Let $B_1, B_2, B_3$ be disjoint sets of sizes: $k$, $k$, $2l-k+t$; let $B = B_1 \cup B_2 \cup B_3$. Then $|A| = |B| = n$. Colour the edges between $A$ and $B$ as follows, where $(i,j)$ denotes the edges between $A_i$ and $B_j$:

        \begin{itemize}
            \item red edges: $(1,1), (2,1), (3,2), (4,2), (5,3)$;
            \item blue edges: $(1,2), (2,3), (3,3), (4,1), (5,1)$;
            \item green edges: $(1,3), (2,2), (3,1), (4,3),(5,2)$.
        \end{itemize}
        
        We note that each red component has at most $k$ vertices in one of the sides, while every green and blue component has at most $l$ vertices on one of the sides, where for the components containing $B_3$ this holds since $|B_3|=2l-k+t \le 2l-k+(k-l)=l\le k$. It follows that there is no red $(k+1)$-\cm, or a blue or green $(l+1)$-\cm{}. In particular, $r(k+1,l+1,l+1) \le k+2l+t$, as required.
    \end{proof}
    \begin{figure}[ht]
        \caption{Illustration of the Example given in Lemma \ref{lem:lower-bound}.}
        \includegraphics[scale=.9]{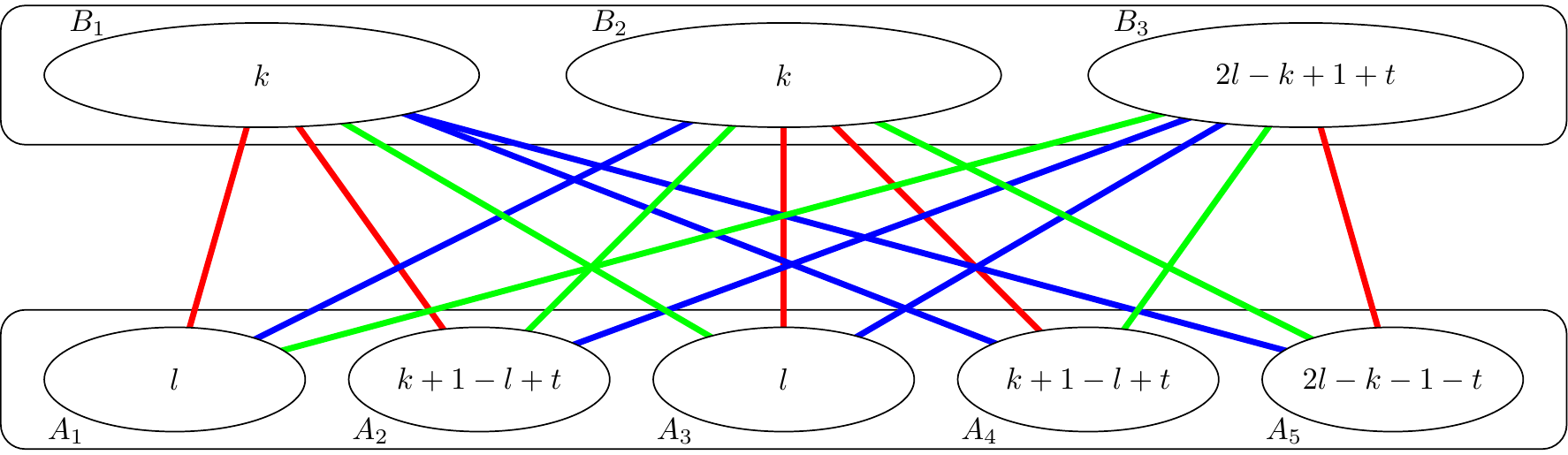}
        \label{fig:eg2}
	\end{figure}

    \begin{cor} \label{cor:asym-k-k-<-l-cm}
        Given $\frac{2}{3}k \le l\le k$, $r(k,l,l)=2k+l-2$.
    \end{cor}

    \begin{proof}
        As $k\ge l$ we have $r(k,l,l) \le r(k,k,l)=2k+l-2,$ where the equality follows from \Cref{cor:asym-k-k->-l-cm}. The assumption $\frac{2}{3}k \le l\le k$ implies we have $t=k-l$ in Lemma \ref{lem:lower-bound} from which it follows that $r(k,l,l) \ge 2k+l-2.$ 
    \end{proof}

	\Cref{cor:asym-k-k->-l-cm,cor:asym-k-k-<-l-cm} determine $r(k,l,l)$ for $l \ge  \frac{2}{3}k$. We are able to determine $r(k,l,l)$ for all $k$ and $l$, using similar methods; we omit the proof but summarise the results here.
	\begin{thm} \label{thm:k-k-l-all}
	    $$r(k,l,l) = 
            \begin{cases}
                k+2l-1 &\quad\text{if } l \le \frac{k+1}{2}\\
                4l-2 &\quad\text{if } \frac{k+1}{2} < l \le \frac{2k}{3}\\
                2k+l-2 &\quad\text{if } \frac{2k}{3} < l < k\\
                k+2l-2 &\quad\text{if } k \le l\\
            \end{cases}
        $$
    \end{thm}
    
    The general asymmetric case of $r(k,l,m)$ remains open, mostly because it is highly unclear what the tight examples should look like. We present the rest of the argument for the symmetric case only to reduce the notational clutter, but the remaining arguments easily generalise to the asymmetric case. In particular, in order to determine (asymptotically) the asymmetric bipartite Ramsey numbers of even cycles it is enough to determine the values of $r(k,l,m).$ 
    
\section{Monochromatic \cms{} in almost complete bipartite graphs}\label{sec:almost-complete}
	
	In this section we generalise the results of the previous section to a setting where the underlying complete bipartite graph is replaced by an almost complete bipartite graph; this makes it suitable for use together with the regularity lemma (see \Cref{sec:regularity}). Interestingly, rather than reproving \Cref{thm:three-col-symm-cm} for almost complete bipartite graphs directly, we reduce the problem for almost complete bipartite graphs to the problem for complete bipartite graphs.
	

    \begin{thm} \label{thm:three-col-symm-cm-almost-complete}
		Let $0< \eps < 8^{-6}$ and $N\ge (3 + 8^6 \eps)n.$ Let $G$ be a subgraph of $K_{N,N}$ of minimum degree at least $N-\eps n.$ Then, in every $3$-colouring of $G$, there is a monochromatic $n$-\cm.
	\end{thm}
    
	\begin{proof}
	    Fix a $3$-colouring of $G$. Our goal is to find a monochromatic $n$-\cm{}. Note that we may assume that $N=\ceil{ (3+8^6\eps)n}$ (if $N$ is larger, just remove some vertices from $G$).
		
		We partition the vertices of $G$ into so-called red \emph{virtual component} $C_{R,1}, \ldots, C_{R,t}$, where for every $i \in [t]$, $C_{R,i}$ is either a red component of order at least $n$, or it is the union of red components of order smaller than $n$ each such that the union has order smaller than $2n$. It is easy to see that we may further assume that all, but at most one, of the virtual components have order at least $n$. It follows that there are at most eight red virtual component, since $N=\ceil{ (3+8^6\eps)n}\le 4n$. We remark that there may be many such partitions; we choose one possible partition arbitrarily and fix it throughout the proof.
		We obtain similar partitions of the vertices into blue virtual component $C_{B, i}$ and green virtual component $C_{G, i}$, where $i \in [8]$.
        
        For each virtual component $C$ we pick a minimum cover set $W$ of $C$. We now add to $G$, in the colour of $C$, all non-edges (with one end in $R$ and one in $L$) that are incident with $W$ inside $C$ and denote the resulting graph by $G_1$. We note that the size of the minimum cover of $C$ stays the same, hence by K\"onig's Theorem, so does the size of the maximum matching in each virtual component. Also note that every virtual component of $G$ is also a virtual component of $G_1$.
        
        To each non-edge of $G_1$ (with one end in $R$ and one in $L$) we assign one of at most $8^6$ types: we say that a non-edge is of type $(a,b,c,d,e,f)$ if its end in $R$ belongs to the virtual component $C_{R,a}, C_{B,b}, C_{G,c}$ and its end in $L$ belongs to the components $C_{R, d}, C_{B, e}, C_{G, f}$.
        
		\begin{claim} \label{claim:cover-bad-edges}
			For any fixed type there is a set of at most $\eps n$ vertices that covers all non-edges of $G_1$ of this type.
		\end{claim}

		\begin{proof}
			Let us fix a type $(a,b,c,d,e,f)$ and let us assume for the sake of contradiction that there is a matching $x_1y_1, \ldots, x_ty_t$ of non-edges of $G_1,$  with $t>\eps n.$ 

			If there are no edges in $G_1$ between $X=\{x_1,\ldots x_t\} \subseteq L$ and $Y=\{y_1,\ldots y_t\} \subseteq R$ then any of the vertices in $X \cup Y$ has more than $\eps n$ non-neighbours in $G_1$, so also in $G,$ a contradiction. So we may assume that there is an edge $x_iy_j$ in $G_1.$ Without loss of generality, we assume that $x_iy_j$ is red, which implies that $a=d$. This in turn implies that either $x_i$ or $y_j$ is in the cover $W$ of $C_{R,a}$ that we used to define $G_1$. But from this it follows that either $x_i y_i$ or $x_j y_j$ is in $G_1$, by definition of $G_1$, a contradiction.
		\end{proof}

		It is now easy to complete the proof of \Cref{thm:three-col-symm-cm-almost-complete}. For each type $(a,b,c,d,e,f)$, fix a cover of the non-edges of this type, as guaranteed by the above claim, and let $U$ be the union of these covers. Note that $U$ is a cover of the non-edges of $G_1$ of size at most $8^6 \eps n$.

		Construct a graph $G_2$ by removing $W$ from the vertex set of $G_1$. $G_2$ is a $3$-coloured complete bipartite graph, with at least $N-8^6\eps n\ge 3n$ vertices on each side. By \Cref{thm:three-col-symm-cm}, $G_2$ contains a, say, red $n$-\cm{} $M$. By construction, $M$ is contained in a red virtual component $C_{R,i}$. In fact, since $M$ spans at least $2n$ vertices, $C_{R,i}$ must be a genuine component in $G$ (rather than a union of components with fewer than $2n$ vertices in total). As $C_{R,i}$ contains a matching of size $n$, its cover number with respect to $G_1$ is at least $n$, which by construction of $G_1$ implies that its cover number with respect to $G$ is also at least $n$. It follows that $G$ contains an $n$-\cm, as required.
	\end{proof}
			
\section{Proof of Theorem \ref{thm:main}}\label{sec:regularity}
	This section is split into two subsections. In the first subsection we introduce the preliminaries that are required for stating the regularity lemma and \L{}uczak's method of \cms{}, and then we proceed to use these tools to complete the proof of \Cref{thm:main}. 
	
    \subsection{Regularity lemma preliminaries}
	    In this subsection we state the regularity lemma and a specific lift lemma which we will need to obtain the desired cycle in the original graph.
	 
        Let us recall some basic definitions related to the regularity lemma. Let $A,B$ be disjoint subsets of vertices in a graph $G$. We denote by $e_G(A,B)$ the number of edges in $G$ with one endpoint in $A$ and one in $B$, and denote the edge density by $d_G(A,B)=\frac{e_G(A,B)}{|A||B|}$. Given $\eps > 0$, we say that the pair $(A,B)$ is $\eps$-\textit{regular} (with respect to the graph $G$) if for every $A'\subseteq A$ and $B' \subseteq B$ satisfying $|A'| \ge \eps |A|$ and $|B'| \ge \eps |B|$ we have
        $$\left| d_G(A',B') -d_G(A,B)\right| < \eps.$$
    
        A partition $\P=\{V_0, V_1,\ldots, V_k\}$ of the vertex set $V$ is said to be $(\eps,k)$-\textit{equitable} if $|V_0|\le \eps |V|$ and $|V_1|=\ldots= |V_k|.$ An $(\eps,k)$-equitable partition $\P$ is $(\eps,k)$-\textit{regular} if all but at most $\eps \binom{k}{2}$ pairs $(V_i,V_j)$ with $1 \le i < j \le k$ are $\eps$-regular. Szemer\'edi's regularity lemma \cite{sze1978-regularity} states that for any $\eps$ and $k_0$ there are $K_0 = K_0(\eps,k_0)$ and $n_0 = n_0(\eps,k_0),$ such that any graph on at least $n_0$ vertices admits an $(\eps,k)$-regular partition with $k_0 \le k \le K_0.$ We shall use the following so-called degree form of the regularity lemma \cite{regularity-survey}, adapted to a $3$-colour and bipartite setting. The version we state here is specifically adapted to our needs, and the adaptations, while perhaps not very standard, are easy to achieve by slightly modifying the usual proof of the regularity lemma.

  	    \begin{lem}\label{lem:mult-col-reg}
            For any $\eps >0$ and $k_0$ there exist $K_0=K_0(\eps,k_0)$ and $n_0=n_0(\eps,k_0),$ such that the following holds. Let $G$ be a $3$-coloured balanced bipartite graph, with bipartition $\{R, L\}$, where $|R| = |L| = n \ge n_0$. Then there exists an $(\eps, 2k)$-equitable partition $\P = \{V_0, \ldots, V_{2k}\}$ of $V(G)$ such that the following properties hold.
            \begin{enumerate}[label=(\alph*), ref=\alph*]
                \item \label{itm:reg-consistent-parts}
                    $V_i$ is contained in $R$ for $1 \le i\le k$ and in $L$ for $k < i\le 2k$.
                \item \label{itm:reg-balanced}
                    $|V_0 \cap R| = |V_0 \cap L|$;
                \item
                    $k_0 \le 2k \le K_0$;
                \item \label{itm:reg-density}
                    for every $i \in [2k]$, for all but at most $2\eps k$ values of $j \in [2k]$, $(V_i, V_j)$ is $\eps$-regular with respect to each of the colours of $G$.
            \end{enumerate}
  	    \end{lem}

        \begin{defn} \label{def:reduced-graph}
            Given an edge-coloured graph $G$, and a partition $\P = \{V_0, \ldots, V_k\}$, the \textit{$(\eps,d)$-reduced graph} $\Gamma$ is the graph whose vertices are $\{V_1, \ldots, V_k\}$ and $V_i V_j$ is an edge if and only if $(V_i, V_j)$ is $\eps$-regular with respect to each colour of $G$ and its density in $G$ is at least $d$. We colour each edge $V_iV_j$ with a majority colour in $G[V_i, V_j]$.
        \end{defn}

        The following lemma is used to lift a connected matching found in the reduced graph to a cycle in the original graph; it was proved by Figaj and {\L}uczak in \cite{figaj-luczak} as a concluding part of their argument for the three colour Ramsey number of even cycles; the idea of using \cms{} this way was introduced by \L{}uczak \cite{luczak99-con-match}.

	    \begin{lem}\label{lem:con-mat-to cycle}
            Given $\eps,d,k$ such that $1 > d > 20\eps > 0$ there is an $n_0$ such that the following holds. Let $\P$ be an $(\eps, k)$-equitable partition of a graph $G$ on $n \ge n_0$ vertices, and let $\Gamma$ be the corresponding $(\eps, d)$-reduced graph. Suppose that $\Gamma$ contains a monochromatic $m$-\cm{}. Then $G$ contains an even cycle of the same colour and of length $\ell$ for every even $\ell \le 2(1 - 9\eps d^{-1})m|V_1|$.
	    \end{lem}	

    \subsection{Main result}
    
        It is now easy to complete the proof of \Cref{thm:main}.
        \begin{proof}[ of \Cref{thm:main}]
            Let $\mu > 0$, let $N = (3 + \mu)n$ and suppose that $n$ is large. Our goal is to show that every $3$-colouring of $K_{N,N}$ contains a monochromatic cycle of length $2n$. Let $\eps > 0$ be sufficiently small. Apply the regularity lemma (\Cref{lem:mult-col-reg}) to the graph $G$ with parameter $\eps$ and let $\P$ be a partition that satisfies the conditions of the lemma. Consider the corresponding $(\eps, 1)$-reduced graph $\Gamma$. Note that by (\ref{itm:reg-consistent-parts}) and (\ref{itm:reg-balanced}) in \Cref{lem:mult-col-reg}, $\Gamma$ is a balanced bipartite graph; denote the number of vertices in each side by $k$ so that $\P = \{V_0, \ldots, V_{2k}\}$. Furthermore, every pair $(V_i, V_j)$, where $V_i, V_j$ are parts of $\P$ in opposite sides of the bipartition, has density $1$ in the original graph. Hence, whenever such $(V_i, V_j)$ is $\eps$-regular with respect to each colour, $V_iV_j$ is an edge in $\Gamma$. It follows from (\ref{itm:reg-density}) that $\Gamma$ has minimum degree at least $(1 - 2\eps)k$.
            
            We would like now to apply \Cref{thm:three-col-symm-cm-almost-complete} to find a large monochromatic \cm{} in $\Gamma$. Let $n' = \frac{k}{3 + 8^7 \eps} \ge k/4$ (the inequality holds because $\eps$ is sufficiently small). As every vertex in one side of $\Gamma$ has at most $2\eps k \le 8\eps n'$ non-neighbours in the other side, \Cref{thm:three-col-symm-cm-almost-complete} implies that $\Gamma$ contains a monochromatic \cm{} of size $n'$. By \Cref{lem:con-mat-to cycle}, $G$ contains a monochromatic cycle of length $\ell$ for any even $\ell \le 2(1 - 9\eps)n'|V_1|$. Note that 
            \begin{align*}
                2(1 - 9\eps)n'|V_1|  
                & = 2(1 - 9\eps)\cdot \frac{k}{3 + 8^7 \eps} \cdot |V_1| \\
                & \ge 2(1 - 9\eps )(1 - \eps) \cdot \frac{N}{3 + 8^7 \eps} \\
                & = 2(1 - 9\eps)(1 - \eps) \cdot \frac{3 + \mu}{3 + 8^7 \eps} \cdot n \ge 2n.
            \end{align*}
            Where the first inequality follows as $k|V_1|=N-|V_0|/2 \ge(1-\eps)N$ and the last inequality holds for $\eps$ sufficiently small compared to $\mu.$ Hence, there is a monochromatic cycle of length $2n$, as required.
        \end{proof}

\section{Concluding remarks and open problems} \label{sec:conc-remarks}

    In this paper we asymptotically determine the $3$-colour bipartite Ramsey number of even cycles and consequently for paths. The most natural next question is to determine what happens for four colours or more, especially as these cases are not even known in the ordinary Ramsey setting and our methods do show some promise.

    Another interesting direction is to try extending our result to hold exactly for large enough cycles or paths, similarly to \cite{kohayakawa-simonovits-skokan,ramsey3path,benevides-skokan}, probably using stability. This raises the question of showing a stability version of our result. One issue with showing this is the following class of examples.
    
    Given $N=3k-3,$ split the vertices of $L(K_{N,N})$ into sets $A_1,A_2,A_3$ such that $|A_i|=k-1$ and split the vertices of $R(K_{N,N})$ into sets $B_1,B_2,B_3$ such that $|B_3| \le k-1.$ Colour all $A_1-B_1$ and $A_2-B_2$ edges red, colour $A_1-B_2$ and $A_2-B_1$ edges blue and colour $A_3-(B_1 \cup B_2)$ and $B_3-(A_1 \cup A_2)$ edges green, and, finally, assign red or blue colours to any edge between $A_3,B_3$ arbitrarily (see Figure \ref{fig:eg3}). This example has no monochromatic $k$-\cm{}, so consequently no monochromatic cycle of length $2k$ or a path of length $2k-1$. It demonstrates that there can be a major proportion of the graph with rather arbitrary assignment of colours.
    
    \begin{figure}[ht]
        \caption{Class of examples for the symmetric case.}
        \includegraphics[scale=.8]{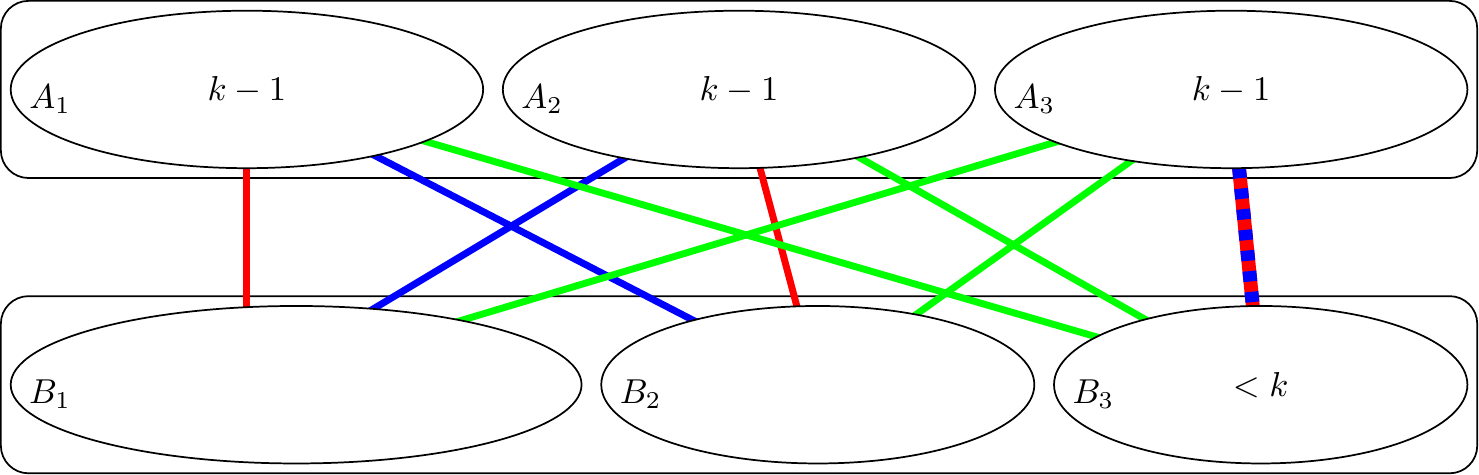}
        \label{fig:eg3}
	\end{figure}

    Another natural direction concerns asymmetric $3$-colour bipartite Ramsey numbers. Note that as Theorem \ref{thm:three-col-symm-cm-almost-complete} and the arguments in \Cref{sec:regularity}, based on regularity lemma, both easily extend to the asymmetric case, the main difficulty is in solving the corresponding asymmetric problem for connected matchings. Recall that $r(k,l,m)$ is the minimum $n$ such that in every $3$-colouring of $K_{n,n}$ there is a red $k$-\cm, a blue $l$-\cm{} or a green $m$-\cm. In \Cref{thm:k-k-l-all}, we determined the value of $r(k,l,l)$ for every $k$ and $l$. Interestingly, \Cref{cor:asym-k-k-<-l-cm} shows that the extremal examples have very different behaviours, depending on the values of $k$ and $l$. This suggests that the behaviour of the $r(k,l,m)$, where $k, l, m$ are allowed to differ, might be very interesting in its own right. 

    Finally we note that our methods above can be used to prove $r_2^{\text{bip}}(C_{2n},C_{2m})=(1+o(1))(n+m),$ for all $n,m$, hence resolving, asymptotically, the conjecture of Zhang, Sun and Wu \cite{zhang2}. It would be interesting to resolve it exactly.
    
    \subsection*{Acknowledgements}
        We would like to thank the referees for their helpful comments; in particular, we are thankful to the referee who suggested a way to simplify the proof of \Cref{thm:three-col-symm-cm-almost-complete}.

    \providecommand{\bysame}{\leavevmode\hbox to3em{\hrulefill}\thinspace}
\providecommand{\MR}{\relax\ifhmode\unskip\space\fi MR }
\providecommand{\MRhref}[2]{%
  \href{http://www.ams.org/mathscinet-getitem?mr=#1}{#2}
}
\providecommand{\href}[2]{#2}

\end{document}